\documentclass{amsart}
\usepackage[cp1250]{inputenc}

\newtheorem{theorem}{Theorem}[section]

\newtheorem{lemma}[theorem]{Lemma}

\newtheorem{definition}[theorem]{Definition}
\newtheorem{example}[theorem]{Example}
\newtheorem{remark}[theorem]{Remark}
\numberwithin{equation}{section}

\begin{document}

\title[Quasi-particle fermionic formulas for $(k,3)$-admissible configurations]
 {Quasi-particle fermionic formulas for $(k,3)$-admissible configurations}

\author{Miroslav Jerkovi\'{c} and Mirko Primc}

\address{University of Zagreb, Zagreb, Croatia}

\curraddr{}

\email{mjerkov@fkit.hr, primc@math.hr}

\subjclass[2000]{Primary 17B67; Secondary 17B69, 05A19.\\ \indent Partially supported by the Ministry of Science and Technology of the Republic of Croatia, Project ID 037-0372794-2806}

\begin{abstract}
We construct new monomial quasi-particle bases of Feigin-Stoya\-nov\-sky's type subspaces for affine Lie algebra $\mathfrak{sl}(3,\mathbb{C})^{\widetilde{}}$ from which the known fermionic-type formulas for $(k,3)$-admissible configurations follow naturally. In the proof we use vertex operator algebra relations for standard modules and coefficients of intertwining operators.
\end{abstract}

\maketitle

\section{Introduction}

Finding combinatorial bases of standard modules for affine Lie algebras is a part of Lepowsky-Wilson's approach \cite{LW} to study Rogers-Ramanujan type identities by using representation theory. This approach lead J.~Lepowsky and R.L.~Wilson to discover vertex operator construction of fundamental $\mathfrak{sl}(2,\mathbb{C})^{\widetilde{}}$-modules and ever since vertex operators are important technical tool in a study of combinatorial bases and fermionic character formulas. Very important role in this study play principal subspaces of standard modules, let us mention only the works of B.~Feigin and A.V.~Stoyanovsky \cite{FS}, G.~Georgiev \cite {G} and C.~Calinescu, S.~Capparelli, J.~Lepowsky and A.~Milas \cite{C1, C2, CalLM1, CalLM2, CalLM3, CLM1, CLM2} which are related to this paper. In \cite{AKS} fermionic characters are obtained for arbitrary highest-weight integrable $\mathfrak{sl}(\ell+1,\mathbb{C})^{\widetilde{}}$-modules.

In this paper we are interested in Feigin-Stoyanovsky's type subspaces of standard modules for affine Lie algebras $\mathfrak{sl}(\ell + 1,\mathbb{C})^{\widetilde{}}$. These subspaces are in many ways parallel to principal subspaces, and by using methods developed by Capparelli, Lepowsky and Milas for principal subspaces, one can construct bases for Feigin-Stoyanovsky's type subspaces of all standard modules (cf. \cite{B1, C1, C2, J1, P2, T1, T2})---in the case of $\mathfrak{sl}(\ell + 1,\mathbb{C})^{\widetilde{}}$ these combinatorial bases are closely related to $(k,\ell +1)$-admissible
configurations studied in \cite{FJLMM, FJMMT1, FJMMT2}. But, in a sharp contrast with this general results, we know fermionic-type character formulas for higher level modules only for affine Lie
algebra $\tilde{\mathfrak{g}}=\mathfrak{sl}(3,\mathbb{C})^{\widetilde{}}$ (cf. \cite{FJMMT1, FJMMT2, J1, J2, J3}).

In order to understand better the combinatorics behind fermionic character formulas for Feigin-Stoyanovsky's type subspaces, in this paper we construct quasi-particle type bases, originally used in
\cite{FS} and \cite{G} for principal subspaces. By following Georgiev's arguments we construct quasi-particle monomial bases of Feigin-Stoyanovsky's type subspaces for $\mathfrak{sl}(3,\mathbb{C})^{\widetilde{}}$ and recover the known fermionic character formula. Our construction also shows why similar arguments cannot work for higher levels and higher ranks, say level two $\mathfrak{sl}(4,\mathbb{C})^{\widetilde{}}$-modules.

One of the main tools applied here is the vertex operator algebra construction of fundamental $\mathfrak{sl}(\ell + 1,\mathbb{C})^{\widetilde{}}$-modules (cf. \cite{FK, S} and also \cite{DL,FLM}) and the use of intertwining operators arising from this construction.

We now present a short overview of the main results.

For a simple complex Lie algebra $\mathfrak{g}$ of type $A_2$ and a Cartan subalgebra $\mathfrak{h}$ of $\mathfrak{g}$ introduce a $\mathbb{Z}$-grading $\mathfrak{g} = \mathfrak{g}_{-1} + \mathfrak{g}_0 + \mathfrak{g}_1$ such that $\mathfrak{h} \subset \mathfrak{g}_0$. Let $\tilde{\mathfrak{g}}$ be the affine Lie algebra associated to $\mathfrak{g}$ with the associated $\mathbb{Z}$-grading.

Choose a basis for $\mathfrak{g}_1$ given by root vectors $x_{\gamma_1}$ and $x_{\gamma_2}$, where for simple roots $\alpha_1$ and $\alpha_2$ the roots $\gamma_1 = \alpha_1 + \alpha_2$ and $\gamma_2 = \alpha_2$ represent the so-called set of colors.

Define Feigin-Stoyanovsky's type subspace of a standard $\tilde{\mathfrak{g}}$-module $L(\Lambda)$ by \begin{align*} W(\Lambda) = U(\tilde{\mathfrak{g}}_1) \cdot v_{\Lambda} \end{align*}
for $\tilde{\mathfrak{g}}_1 = \mathfrak{g}_{1}\otimes\mathbb{C}[t,t^{-1}]$ and $v_{\Lambda}$ being a highest weight vector in $L(\Lambda)$.

Define a quasi-particle $x_{n\gamma}(m)$ of color $\gamma \in \{ \gamma_1, \gamma_2 \}$, charge $n$ and degree $m$: \begin{align*} x_{n\gamma_i}(m) = \sum_{\genfrac{}{}{0pt}{}{m_1,\dots, m_n
\in \mathbb{Z}}{m_1+\dots + m_n = m}} x_{\gamma_i} (m_n) \cdots x_{\gamma_i}(m_1), \end{align*} where by $x(m) = x\otimes t^m$ for $x \in \mathfrak{g}$ and $m \in \mathbb{Z}$ we denote the elements of $\tilde{\mathfrak{g}}$.

Let $\Lambda_0$, $\Lambda_1$ and $\Lambda_2$ be the fundamental weights of $\tilde{\mathfrak{g}}$. One can show that for integral dominant weights $\Lambda = k_0\Lambda_0 + k_1\Lambda_1$ and
$\Lambda = k_1\Lambda_1 + k_2\Lambda_2$ one can parameterize a basis $\frak{B}_{W(\Lambda)}\cdot v_\Lambda$ for $W(\Lambda)$ by quasi-particle monomials \begin{align*} x_{n_{2,b} \gamma_2} (m_{2,b}) \cdots x_{n_{2,1}\gamma_2}(m_{2,1}) x_{n_{1,a}\gamma_1}(m_{1,a}) \cdots x_{n_{1,1}\gamma_1}(m_{1,1}) \end{align*} with degrees $m_{1,1}, \dots, m_{1,a}$ and $m_{2,1}, \dots, m_{2,b}$ obeying certain
initial and difference conditions, cf. definition \eqref{basis} and Theorem~\ref{theorem}.

The spanning set argument for $\frak{B}_{W(\Lambda)}$ is shown by reducing the spanning set of Poincar\'{e}-Birkhoff-Witt type by relations on the set of quasi-particle monomials, see Lemma~\ref{lem_relations 3}, Lemma~\ref{lem_relations 4} and Remark~\ref{R:the way we use relations}. 

The proof of linear independence of $\frak{B}_{W(\Lambda)}$ is carried out by induction on the linear order on quasi-particle monomials, using certain coefficients of intertwining operators, as
well as other operators obtained from the vertex operator algebra construction of fundamental $\tilde{\mathfrak{g}}$-modules. The other main ingredient in the proof of linear independence is a
projection $\pi$ (mimicking the construction in \cite{G}) distributing the constituting quasi-particles of monomials among factors of the tensor product of $\mathfrak{h}-$weight subspaces of
fundamental $\tilde{\mathfrak{g}}$-modules, thus ensuring the compatibility of the usage of above mentioned operators with the defining conditions of $\frak{B}_{W(\Lambda)}$.

Finally, as a direct consequence of having obtained the quasi-particle bases $\frak{B}_{W(\Lambda)}$, one may straightforwardly write down fermionic-type character formulas for $W(\Lambda)$, cf. Theorem~\ref{formulas}.

We also show these formulas are the same as the ones already obtained much less explicitly in \cite{J3} by solving a system of recurrence relations for characters. This result justifies our belief that the quasi-particle approach to building combinatorial bases for $W(\Lambda)$ and to calculating the corresponding fermionic-type characters is worth studying.

\section{Affine Lie algebra $\mathfrak{sl}(\ell + 1,\mathbb{C})^{\widetilde{}}$}

\subsection{Preliminaries}

In this paper we are interested mainly in affine Lie algebra $\mathfrak{sl}(3,\mathbb{C})^{\widetilde{}}$, but it will be convenient to have a bit more general notation.

For $\mathfrak{g} = \mathfrak{sl}(\ell + 1,\mathbb{C})$ let $\mathfrak{h}$ be its Cartan subalgebra, $R$ the corresponding root system with fixed simple roots $\alpha_1, \dots, \alpha_{\ell}$, and
 $\omega_1, \dots, \omega_{\ell}$ the corresponding fundamental weights of $\mathfrak{g}$. For later use we set $\omega_0 = 0$. Denote by $Q=Q(R)$ the root lattice and by $P=P(R)$ the weight lattice of $\mathfrak{g}$, and by $\langle \cdot,\cdot\rangle$ the usually normalized Killing form identifying $\mathfrak{h}$ with $\mathfrak{h}^{\star}$. We have the root space decomposition $\mathfrak{g} = \mathfrak{h} + \sum_{\alpha \in R} \mathfrak{g}_{\alpha}$ and we fix root vectors $x_{\alpha}$.

Let $\tilde{\mathfrak{g}}$ be the associated affine Lie algebra \begin{align*} \tilde{\mathfrak{g}} = \mathfrak{g} \otimes \mathbb{C}[t,t^{-1}] \oplus \mathbb{C} c \oplus \mathbb{C} d, \end{align*} c denoting the canonical central element and d the degree operator, with Lie product given in the usual way (cf. \cite{Kac}). For $x \in \mathfrak{g}$ we write $x(m) = x\otimes t^m$ for $m \in \mathbb{Z}$ and define the formal Laurent series $x(z)= \sum_{m \in \mathbb{Z}} x(m)z^{-m-1}$ in indeterminate $z$.

\subsection{Standard modules}

Denote by $\Lambda_0, \dots, \Lambda_{\ell}$ the corresponding fundamental weights of $\tilde{\mathfrak{g}}$. For a given integral dominant weight $\Lambda = k_0\Lambda_0 + \cdots + k_{\ell} \Lambda_{\ell}$ denote by $L(\Lambda)$ the standard $\tilde{\mathfrak{g}}$-module with highest weight $\Lambda$, and by $v_{\Lambda}$ a fixed highest weight vector of $L(\Lambda)$. Let $k=\Lambda(c)=k_0+\cdots + k_{\ell}$ be the level of $L(\Lambda)$. For a higher level standard module $L(\Lambda)$ we will use the Kac theorem on complete reducibility to embed it in the corresponding tensor product of fundamental modules, \begin{align*}L(\Lambda)\subset L(\Lambda_{\ell})^{\otimes k_{\ell}} \otimes \cdots \otimes L(\Lambda_1)^{\otimes k_1} \otimes L(\Lambda_0)^{\otimes k_0},\end{align*} and we use a highest weight vector \begin{align*} v_{\Lambda} = v_{\Lambda_{\ell}}^{\otimes k_{\ell}} \otimes \cdots \otimes v_{\Lambda_1}^{\otimes k_1} \otimes v_{\Lambda_0}^{\otimes k_0}. \end{align*}

\subsection{VOA construction for fundamental modules}

We now recall the basic facts about the Frenkel-Kac vertex operator algebra construction \cite{FK,S} of level one $\tilde{\mathfrak{g}}$-modules $L(\Lambda_0), \dots, L(\Lambda_{\ell})$, for our notation and details see \cite{DL, FLM}.

Define $V_P= M(1)\otimes \mathbb{C}[P]$, a module for $\hat{\mathfrak{h}} = \mathfrak{h}\otimes \mathbb{C}[t,t^{-1}] \oplus \mathbb{C} c$, with $M(1)=U(\hat{\mathfrak{h}}) \otimes_{U(\mathfrak{h} \otimes \mathbb{C}[t] \oplus \mathbb{C}c)} \mathbb{C}$ being the Fock space and $\mathbb{C}[P]$ the group algebra of the weight lattice $P$ with basis $\{ e^{\lambda} \mid \lambda \in P\}$.

Define also $V_Q=M(1)\otimes \mathbb{C}[Q]$, with $\mathbb{C}[Q]$ being the group algebra of the root lattice. There is a natural structure of simple vertex operator algebra on $V_Q$, with vertex operators which may be defined on $V_P$: \begin{align} \label{vertex} Y(1 \otimes e^{\lambda},z)=E^{-}(-\lambda, z) E^{+}(-\lambda, z)\otimes e^{\lambda}z^{\lambda}\epsilon_{\lambda}, \ \ \lambda \in P, \end{align} where $E^{\pm}(\lambda, z) = \exp \Big( \sum_{m \geq 1} \lambda (\pm m) \frac{z^{\mp m}}{\pm m} \Big)$, and for $\lambda, \mu \in P$ let $z^{\lambda} e^{\mu} = z^{\langle \lambda , \mu \rangle} e^{\mu}$, $\epsilon_{\lambda}e^{\mu}= \epsilon(\lambda, \mu) e^{\mu}$, where $\epsilon$ is a certain 2-cocyle on $P$ (cf. \cite{DL, FLM}).

Using \eqref{vertex}, one can extend the action of $\hat{\mathfrak{h}}$ on $V_P$ to the action of $\tilde{\mathfrak{g}}$ by letting $x_{\alpha} \otimes t^m$, $\alpha \in R$, $m \in \mathbb{Z}$, act as a corresponding coefficient of $Y(1\otimes e^{\alpha}, z)$ via \begin{align} \label{gaction} x_{\alpha}(z) = \sum_{m \in \mathbb{Z}} x_{\alpha}(m)z^{-m-1} = Y(1\otimes e^{\alpha}, z). \end{align} This action then gives $V_Qe^{\omega_i} \cong L(\Lambda_i)$ with highest weight vectors $v_{\Lambda_i} = 1 \otimes e^{\omega_i}$, $i=0, \dots, \ell$, and $V_P = \oplus_{i=0}^{\ell} L(\Lambda_i)$.

By using vertex operator formula \eqref{vertex} it is easy to see that on fundamental $\tilde{\mathfrak{g}}$-modules we have relations \begin{align*} x_\alpha(z)x_\beta(z)=0\quad\text{if}\quad\langle\alpha,\beta\rangle\geq 1. \end{align*} By using tensor product of $k$ fundamental modules we obtain VOA relations for level $k$ standard $\tilde{\mathfrak{g}}$-modules, see \eqref{kplusone} bellow.

For $\alpha \in R$ and $\lambda \in P$ \eqref{vertex} gives \begin{align}\label{commutator00} x_\alpha(z)e\sp\lambda= \epsilon(\alpha,\lambda)\, z\sp{\langle \alpha, \lambda \rangle}e\sp\lambda x_\alpha(z). \end{align} By comparing coefficients we get \begin{align*} x_{\alpha}(m)e^{\lambda}& = \epsilon(\alpha,\lambda)e^{\lambda} x_{\alpha} (m + \langle \alpha, \lambda \rangle). \end{align*}

\section{Feigin-Stoyanovsky's type subspaces $W(\Lambda)$}

\subsection{Definition}

In this section we introduce the main object of interest, i.e., a certain vector subspace of given standard $\tilde{\mathfrak{g}}$-module $L(\Lambda)$ with highest weight $\Lambda = k_0 \Lambda + \dots + k_{\ell}\Lambda_{\ell}$.

For the fixed minuscule weight $\omega = \omega_{\ell}$ we have $\mathfrak{g}$ with $\mathbb{Z}$-grading \begin{align*}\mathfrak{g} = \mathfrak{g}_{-1} + \mathfrak{g}_0 + \mathfrak{g}_1 =\sum_{\omega(\alpha) = -1} \mathfrak{g}_{\alpha} + \Big( \mathfrak{h} + \sum_{\omega(\alpha)=0} \mathfrak{g}_{\alpha} \Big)+ \sum_{\omega(\alpha) = 1} \mathfrak{g}_{\alpha}. \end{align*} We then also have the associated $\mathbb{Z}$-grading on $\tilde{\mathfrak{g}}$, \begin{align*} \tilde{\mathfrak{g}} = \tilde{\mathfrak{g}}_{-1} + \tilde{\mathfrak{g}}_0 + \tilde{\mathfrak{g}}_1 = \mathfrak{g}_{-1}\otimes\mathbb{C}[t,t^{-1}] + (\mathfrak{g}_0 \otimes \mathbb{C}[t,t^{-1}]\oplus \mathbb{C} c \oplus \mathbb{C}d) + \mathfrak{g}_{1}\otimes\mathbb{C}[t,t^{-1}]. \end{align*} Note that $\tilde{\mathfrak{g}}_1$ is a commutative subalgebra and a $\tilde{\mathfrak{g}}_0$-module. The set of roots \begin{align*}\Gamma= \{ \alpha \in R \mid \omega(\alpha)=1 \} = \{ \gamma_1, \gamma_2, \dots , \gamma_{\ell} \mid \gamma_i= \alpha_i + \dots + \alpha_{\ell} \} \end{align*} we shall also call the set of colors. Note that $\langle\gamma_i,\gamma_j\rangle=1+\delta_{ij}$ and that $\tilde{\mathfrak{g}}_1$ is spanned by elements of $\tilde{\mathfrak{g}}$ of the form $x_{\gamma} (n)$, $\gamma \in \Gamma$, $n \in \mathbb{Z}$. Later on we shall call $x_{\gamma} (n)$ a particle of color $\gamma$, charge one and degree $n$.

\begin{definition}
For a standard $\tilde{\mathfrak{g}}$-module $L(\Lambda)$ and a given $\mathbb{Z}$-grading on $\tilde{\mathfrak{g}}$ we define its Feigin-Stoyanovsky's type subspace $W(\Lambda)$ as \begin{align*} W(\Lambda) = U(\tilde{\mathfrak{g}}_1) \cdot v_{\Lambda}.\end{align*}
\end{definition}

Poincar\'{e}-Birkhoff-Witt theorem implies that $W(\Lambda)$ is spanned by the following set of particle monomials acting on highest weight vector $v_{\Lambda}$: \begin{align} \label{span} \{ b \cdot v_{\Lambda} \ | \ b = \dots x_{\gamma_1}(-2)^{a_{\ell}} x_{\gamma_{\ell}}(-1)^{a_{\ell-1}} \cdots x_{\gamma_1}(-1)^{a_0},\ a_i \in \mathbb{Z}_{+}, \ i \in \mathbb{Z}_{+}\}. \end{align}

\subsection{Bases in admissible configurations}

By using VOA relations on level $k$ standard $\tilde{\mathfrak{g}}$-modules \begin{align} \label{kplusone} x_{\beta_1}(z) \cdots x_{\beta_{k+1}}(z)=0, \ \ \beta_1, \dots, \beta_{k+1} \in \Gamma,
\end{align} \eqref{span} can be reduced to combinatorial basis for $W(\Lambda)$ in monomial vectors given by action of so-called $(k,\ell+1)$-admissible monomials on $v_{\Lambda}$ (cf. \cite{FJLMM,
P1, P2}):

\begin{definition}
A sequence of non-negative integers $(a_i)_{i=0}^{\infty}$ with finitely many non-zero terms will be called a configuration. We call $(a_i)_{i=0}^{\infty}$ a $(k, \ell+1)$-admissible configuration (for $\Lambda$) if the so-called initial conditions \begin{align} \nonumber a_0 &\leq k_0 \\ \label{initial} a_0+a_1 &\leq k_0 + k_1 \\ \nonumber \dots \\ \nonumber a_0 + a_1 + \dots + a_{\ell-1} &\leq k_0 + \dots + k_{\ell-1}, \end{align} and difference conditions \begin{align} \label{difference} a_i + \dots + a_{i+\ell} \leq k, \quad i \in \mathbb{Z}_{+} \end{align} are met. In this case a monomial $b$ of form \eqref{span}, as well as the accompanying monomial vector $b \cdot v_{\Lambda}$ are said to be $(k, \ell+1)$-admissible.
\end{definition}

\begin{theorem}
\label{thm_basis} Basis for $W(\Lambda)$ is given by $(k, \ell+1)$-admissible monomial vectors.
\end{theorem}

\subsection{Initial conditions}

We elaborate a bit further on initial conditions \eqref{initial} for $\Lambda$. Using \eqref{gaction} it is not hard to calculate that \begin{align} \label{init} & x_{\gamma_i}(-1)v_{\Lambda_j} = \left\{ \begin{array}{l} 0 \quad \text{if} \quad i \leq j, \\ Ce^{\gamma_i}v_{\Lambda_j}\neq 0, \ C \in \mathbb{C}^{\times}\quad \text{if} \quad i > j \end{array}\right. \end{align} and \begin{align} \label{init2} x_{\gamma_i}(-2)v_{\Lambda_j} = C'e^{\gamma_i}v_{\Lambda_j}\neq 0, \ C' \in \mathbb{C}^{\times}\quad \text{if} \quad i \leq j, \end{align} which is a consequence of the fact that $\langle \gamma_i, \omega_j \rangle = 0$ if and only if $i > j$. Stated otherwise, for particle of color $\gamma_i$ and charge one we may find the maximal degree $d_{\max}(\gamma_i, \Lambda_j)$ such that $x_{\gamma_i}(d_{\max}(\gamma_i, \Lambda_j))v_{\Lambda_j}\neq 0$: \begin{align} \label{dmaxone} d_{\max}(\gamma_i, \Lambda_j) = \left\{ \begin{array}{l} -2 \quad \text{if} \quad i \leq j, \\ -1 \quad \text{if} \quad i > j. \end{array}\right.
\end{align}

Using \eqref{kplusone} we get $x_{\beta_2}(-1)x_{\beta_1}(-1)v_{\Lambda_j}=0$ for $j = 0, \dots, \ell$, which taken together with \eqref{init} give $\dots x_{\gamma_{\ell}}(-1)^{a_{\ell-1}} \cdots x_{\gamma_1}(-1)^{a_0} v_{\Lambda} \neq 0$ if and only if \eqref{initial} holds.

\section{Quasi-particle bases of $W(\Lambda)$ for $\mathfrak{sl}(3,\mathbb{C})^{\widetilde{}}$}

In a sharp contrast with the above mentioned general construction of monomial bases for arbitrary $W(\Lambda)$, we know fermionic-type character formulas for higher level modules only for affine Lie
algebra $\tilde{\mathfrak{g}}=\mathfrak{sl}(3,\mathbb{C})^{\widetilde{}}$ (cf. \cite{FJMMT1, FJMMT2, J1, J2, J3}). In fact, we have not been able to guess a correct formula even for the level two subspace $W(2\Lambda_0)$ for $\tilde{\mathfrak{g}}=\mathfrak{sl}(4,\mathbb{C})^{\widetilde{}}$. In this section we propose building new combinatorial bases of $W(\Lambda)$ for $\mathfrak{sl}(3,\mathbb{C})^{\widetilde{}}$ in so-called quasi-particles (cf. \cite{FS, G}) from which the known fermionic-type formulas follow naturally. Our arguments are parallel to Georgiev's and, as in \cite{G}, they work only for certain classes of standard modules---in our case for $\Lambda = k_0 \Lambda_0 + k_1 \Lambda_1$ or $\Lambda = k_1 \Lambda_1 + k_2 \Lambda_2$. We hope that this construction will give us a better insight of what kind of fermionic-type formulas for Feigin-Stoyanovsky's type subspaces we may or we may not expect in higher rank cases.

\subsection{Definition of quasi-particles}

For $i=1,2$ define a quasi-particle of color $\gamma_i$, charge $n$ and degree $m$ by \begin{align*} x_{n\gamma_i}(m) = \sum_{\genfrac{}{}{0pt}{}{m_1,\dots, m_n \in \mathbb{Z}}{m_1+\dots + m_n = m}}
x_{\gamma_i} (m_n) \cdots x_{\gamma_i}(m_1). \end{align*} The corresponding generating function will be denoted by \begin{align*} x_{n\gamma_i}(z) = x_{\gamma_i}(z)^n = \sum_{m \in \mathbb{Z}} x_{n \gamma_i}(m) z^{-m-n}. \end{align*} We assume $n \leq k$ because VOA relation \eqref{kplusone} implies that quasi-particles with charge $n\geq k+1$ are zero on any standard level $k$ module.

\subsection{Initial conditions for quasi-particles}

For $\Lambda = k_0 \Lambda_0 + k_1 \Lambda_1$ or $k_1 \Lambda_1 + k_2 \Lambda_2$ set \begin{align} \label{dmaxmore} d_{\max}(n\gamma_1, \Lambda)&= -\min \{ n, k_0 \} - 2\max \{ n- k_0,0 \}, \\ \nonumber d_{\max}(n\gamma_2, \Lambda) & = - \min \{ n, k_0 + k_1 \} - 2\max \{ n - (k_0 + k_1), 0 \}. \end{align} By using \eqref{dmaxone} it is easy to see that $m>d_{\max}(n\gamma_i, \Lambda)$ implies \
$x_{n\gamma_i}(m)v_\Lambda=0$.

\subsection{Linear order among quasi-particles}

We define the following linear order $\prec$ among quasi-particles: from the usual order on $R$ we have $\gamma_2 < \gamma_1$, and then state that $x_{n_2\beta_2}(m_2) \prec x_{n_1\beta_1}(m_1)$ if one of the following conditions holds:
\begin{itemize}
\item[(1)] $\beta_2 < \beta_1$,
\item[(2)] $\beta_2 = \beta_1$, $n_2 < n_1$,
\item[(3)] $\beta_2 = \beta_1$, $n_2 = n_1$, $m_2 < m_1$.
\end{itemize}
Therefore, two quasi-particles are compared first by color, then by charge, and finally by degree.

Next, extend the order $\prec$ to quasi-particle monomials. First, organize a quasi-particle monomial so that its commuting quasi-particles are written starting from right in a non-increasing manner. For two quasi-particle monomials then set \begin{align*} N = N_{n_2}\cdots N_{1} \prec M=M_{n_1}\cdots M_{1} \end{align*} if there exists $i_0 \in \mathbb{N}$ such that $N_{i} = M_{i}$ for all $i < i_0$, and either $N_{i_0} \prec M_{i_0}$ or $i_0 = n_1 + 1 \leq n_2$. This linear order agrees with multiplication of quasi-particle monomials, i.e. for given quasi-particle monomial $K$ the inequality $N\prec M$ implies $NK\prec MK$.

\begin{definition}
\label{cct} For a quasi-particle monomial \begin{align*} x_{n_{2,b} \gamma_2} (m_{2,b}) \cdots x_{n_{2,1}\gamma_2}(m_{2,1}) x_{n_{1,a}\gamma_1}(m_{1,a}) \cdots x_{n_{1,1}\gamma_1}(m_{1,1})
\end{align*} we will say it is of color-charge type $(n_{2,b},\dots, n_{2,1}; n_{1,a}, \dots, n_{1,1})$. Also, as in \cite{G}, it may be convenient to also use a notion of color-dual-charge type \begin{align*} (r_2\sp{(1)}, r_2\sp{(2)},\dots, r_2\sp{(k)}; r_1\sp{(1)}, r_1\sp{(2)},\dots, r_1\sp{(k)}), \end{align*} where $r_i\sp{(1)}\geq r_i\sp{(2)}\geq \dots \geq r_i\sp{(k)}$ and for $i=1,2$ the number of quasi-particles of color $\gamma_i$ and charge $n$ equals $r_i\sp{(n)}-r_i\sp{(n+1)}$.
\end{definition}

\subsection{Relations among quasi-particles}

Here we establish relations among quasi-particles which will help us reduce to a basis the spanning set for $W(\Lambda)$ consisting of quasi-particle monomial vectors \begin{align}\label{quasi-particle monomial vectors} x_{n_{2,b} \gamma_2} (m_{2,b}) \cdots x_{n_{2,1}\gamma_2}(m_{2,1}) x_{n_{1,a}\gamma_1}(m_{1,a}) \cdots x_{n_{1,1}\gamma_1}(m_{1,1})v_\Lambda. \end{align}

\begin{lemma}
\label{lem_relations 1} Let color $\gamma_i$, $i\in\{1,2\}$, be fixed and $n_2 \leq n_1$. Then for $N=0,\dots, 2n_2-1$ we have relations among generating functions of quasi-particles of the form \begin{align} \label{same} & \Big( \frac{d^N}{dz^N} x_{n_2 \gamma_i}(z) \Big) x_{n_1 \gamma_i} (z) = A_N(z) x_{(n_1+1)\gamma_i}(z) + B_N(z) \frac{d^N}{dz^N} x_{(n_1+1)\gamma_i}(z) \end{align} for some formal series $A_N(z)$ and $B_N(z)$ with coefficients in the set of quasi-particle polynomials.
\end{lemma}

\begin{proof}
Considering \eqref{same}, it suffices to show that for every $N=0,\dots, 2n_2 -1$ each summand in the expansion of $\frac{d^N}{dz^N} x_{n_2 \gamma_i}(z)$ contains either $x_{\gamma_i}(z)$ or $\frac{d}{dz} x_{\gamma_i}(z)$, and use the fact that \begin{align*} \Big( \frac{d}{dz} x_{\gamma_i}(z) \Big) \cdot x_{n_1 \gamma_i} (z) = \frac{1}{n_1 + 1} \frac{d}{dz} x_{(n_1 + 1)\gamma_1}(z).
\end{align*} But, since $x_{n_2 \gamma_i}(z) = x_{\gamma_i}(z)^{n_2}$, differentiating this product $N$ times corresponds to distributing derivation signs among its $n_2$ factors, which means that the highest possible $N$ for which there would not exist a summand with each factor differentiated more than once is $N = 2n_2 - 1$.
\end{proof}

\begin{lemma}
\label{lem_relations 2} For differently colored quasi-particles $x_{n_1 \gamma_1}(z)$ and $x_{n_2 \gamma_2}(z)$ such that $n_1 + n_2 \geq k+1$ we have for $N=0,\dots, n_1+n_2-k-1$ the following relations: \begin{align} \label{different} \Big( \frac{d^N}{dz^N} x_{n_2 \gamma_2}(z) \Big) \cdot x_{n_1 \gamma_1} (z)=0. \end{align}
\end{lemma}

\begin{proof}
Equality \eqref{different} is a consequence of relations \eqref{kplusone}. Similarly as in discussing \eqref{same}, it is easy to see that $N=n_1+n_2-k-1$ is the highest $N$ for which each summand in the expansion of the left-hand side of \eqref{different} has $k+1$ or more particles in product, because for this $N$ the least plausible way of distributing the derivation signs among factors of $x_{n_2 \gamma_2}(z)$ leaves $n_2- (n_1+n_2-k-1) = k+1 - n_1$ factors $x_{\gamma_2}(z)$ undifferentiated, which together with $x_{n_1 \gamma_1}(z)$ produces an equation of the form \eqref{kplusone}.
\end{proof}

\begin{lemma}
\label{lem_relations 3} Let color $\gamma_i$, $i\in\{1,2\}$, be fixed and $n_2 \leq n_1$. Let $r= 2n_2-1$ and let $m_1$ and $m_2$ be a pair of integers. Then any of $r+1$ monomials in the set \begin{align*} \mathcal D=\{x_{n_2 \gamma_i}(m_2)x_{n_1 \gamma_i}(m_1), \dots, x_{n_2 \gamma_i}(m_2-r)x_{n_1 \gamma_i}(m_1+r)\} \end{align*} can be expressed in terms of monomials from the set \begin{align} \label{all without unknowns} \{x_{n_2 \gamma_i}(m_2-j)x_{n_1 \gamma_i}(m_1+j)\mid j\in\mathbb Z\}\,\backslash\, \mathcal D \end{align} and monomials which have as a factor quasi-particle of color $\gamma_i$ and charge $n_1+1$.
\end{lemma}

\begin{proof}
By Lemma~\ref{lem_relations 1} for $N=0,\dots, r$ we have $2n_2$ relations \begin{align*} & \sum_{j_2 \in \mathbb{Z}} {-j_2 - n_2 \choose N } x_{n_2 \gamma_i}(j_2) z^{-j_2
- n_2 - N} \cdot \sum_{j_1 \in \mathbb{Z}} x_{n_1 \gamma_i}(j_1) z^{-j_1 - n_1} \\ & = \sum_{m \in \mathbb{Z}} \,\sum_{j_1 + j_2 = m} {-j_2 - n_2 \choose N } x_{n_2 \gamma_i}(j_2) x_{n_1 \gamma_i}(j_1) z^{-m-n_1 - n_2 - N}\sim 0, \end{align*} where $\sim 0$ means that the left-hand side can be expressed in terms of quasi-particle monomials greater with respect to order $\prec$. So for fixed $m_1, m_2 \in \mathbb{Z}$ we may write down a linear system of $2n_2=r+1$ equations in $r+1$ ``unknown'' quadratic quasi-particle monomials \begin{align*} x_{n_2 \gamma_i}(m_2)x_{n_1 \gamma_1}(m_1), \dots, x_{n_2 \gamma_i}(m_2-r)x_{n_1 \gamma_1}(m_1+r), \end{align*} which should be expressed in terms of higher quasi-particle monomials and other quasi-particle monomials from the set \eqref{all without unknowns}. In order for such an expression to uniquely exist, the described system must be regular, which is checked by showing its matrix (with $p=-m_2-n_2$) \begin{align*} A_{p,r} = \left(\begin{array}{ccccc} {p \choose 0} & {p+1 \choose 0} & {p+2 \choose 0} & \dots & {p+r \choose 0} \\ {p \choose 1} & {p+1 \choose 1} & {p+2 \choose 1} & \dots & {p+r \choose 1} \\ \vdots & \vdots & \vdots & \ddots & \vdots \\ {p \choose r} & {p+1 \choose r} & {p+2 \choose r} & \dots & {p+r \choose r} \\ \end{array} \right) \end{align*} is regular: starting from the second one, we subtract from each column of $\det A_{p,r}$ the previous one, and using ${a+1 \choose b+1} - {a \choose b+1} = {a \choose b}$ we easily get \begin{align*} \det A_{p,r} = \det A_{p, r-1} = \dots = \det A_{p,0} = 1. \end{align*}
\end{proof}

In a similar way we prove the following lemma by using Lemma~\ref{lem_relations 2}:

\begin{lemma}
\label{lem_relations 4} For differently colored quasi-particles with charges $n_1 \gamma_1$ and $n_2 \gamma_2$ such that $n_1 + n_2 \geq k+1$ set $r=n_1+n_2-k-1$. Then for any pair of integers $m_1$ and $m_2$ any of $r+1$ monomials in the set \begin{align*} \mathcal D=\{x_{n_2 \gamma_2}(m_2)x_{n_1 \gamma_1}(m_1), \dots, x_{n_2 \gamma_2}(m_2-r)x_{n_1 \gamma_1}(m_1+r)\} \end{align*} can be expressed in terms of monomials from the set \begin{align*} \{x_{n_2 \gamma_2}(m_2-j)x_{n_1 \gamma_1}(m_1+j)\mid j\in\mathbb Z\}\,\backslash\, \mathcal D. \end{align*}
\end{lemma}

\begin{remark}
\label{R:the way we use relations}
We shall use Lemmas~\ref{lem_relations 3} and \ref{lem_relations 4} in reducing spanning set \eqref{quasi-particle monomial vectors} in two ways:
\begin{itemize}
\item[(1)] In the case when two quasi-particles have the same color $\gamma_i$ and the same charge $n_1=n_2=n$, we express $2n$ quadratic quasi-particle monomials \begin{align*} &x_{n \gamma_i}(j-n)x_{n \gamma_i}(j+n),\dots, x_{n \gamma_i}(j)x_{n \gamma_i}(j),\dots, x_{n \gamma_i}(j+n-1)x_{n \gamma_i}(j-n+1)\ \text{or}\\ &x_{n \gamma_i}(j-n)x_{n \gamma_i}(j-1+n),\dots, x_{n \gamma_i}(j)x_{n \gamma_i}(j-1),\dots, x_{n \gamma_i}(j+n-1)x_{n \gamma_i}(j-n) \end{align*} in terms of bigger quasi-particle monomials and monomials satisfying difference $2n$-condition: \begin{align} \label{E:difference 2n condition for same} x_{n \gamma_i}(m_2)x_{n \gamma_i}(m_1) \quad\text{such that}\quad m_1\leq m_2-2n. \end{align}
\item[(2)] It may happen that quasi-particle monomial vector \eqref{quasi-particle monomial vectors} contains as factors two quasi-particles of the same color $\gamma_i$ and different charges $n_2<n_1$, say \begin{align*}x_{n_2\gamma_i}(m_2)x_{n_1 \gamma_i}(m_1).\end{align*} By initial conditions there is $d_2$ such that \begin{align*} x_{n_2\gamma_i}(m_2)v_\Lambda\neq0 \quad\text{implies}\quad m_2\leq d_2. \end{align*} However, the presence of quasi-particle $x_{n_1 \gamma_i}(m_1)$ allows us to use Lemma~\ref{lem_relations 3} for the sequence \begin{align*}x_{n_2\gamma_i}(d_2)x_{n_1 \gamma_i}(m_1-2n_2),\dots,x_{n_2\gamma_i}(d_2-2n_2+1)x_{n_1 \gamma_i}(m_1-1).\end{align*} Hence, by using order $\prec$, we may assume that our quasi-particle monomial vector \eqref{quasi-particle monomial vectors} contains only factors \begin{align} \label{E:difference 2n condition for different}x_{n_2\gamma_i}(m_2)x_{n_1 \gamma_i}(m_1)\quad\text{such that}\quad m_2\leq d_2-2n_2. \end{align} In a similar way we also use Lemma~\ref{lem_relations 4}.
\end{itemize}
\end{remark}

\subsection{Quasi-particle bases}

We start by a definition of a set $\frak{B}_{W(\Lambda)}$:

\begin{definition}
\begin{align} \label{basis} \frak{B}_{W(\Lambda)} = \hspace{-8mm} & \bigsqcup_{\begin{array}{c}{\scriptstyle 0 \leq n_{1,a} \leq \cdots \leq n_{1,1}\leq k}\\{\scriptstyle 0 \leq
n_{2,b} \leq \cdots \leq n_{2,1}\leq k}\end{array}} \hspace{-3mm} \left\{ x_{n_{2,b} \gamma_2} (m_{2,b}) \cdots x_{n_{2,1}\gamma_2}(m_{2,1}) x_{n_{1,a}\gamma_1}(m_{1,a}) \cdots x_{n_{1,1}\gamma_1}(m_{1,1}) \begin{array}{c}\\ \\ \\ \\ \end{array} \right| \\ \nonumber & \left|\begin{array}{lll} m_{1,j+1} \leq m_{1,j} - 2n_{1,j} \ \ \textrm{if} \ \ n_{1,j+1} = n_{1,j} & \\
m_{2,j+1} \leq m_{2,j} - 2n_{2,j} \ \ \textrm{if} \ \ n_{2,j+1} = n_{2,j} & \\ m_{1,j} \leq d_{\max}(n_{1,j}\gamma_1, \Lambda) - \sum_{j^{\prime} < j} 2n_{1,j} & \\ m_{2,j} \leq d_{\max}(n_{2,j}\gamma_2, \Lambda) - \sum_{j^{\prime} < j } 2n_{2,j} - \sum_{j^{\prime}} \max \{ 0, n_{2,j} + n_{1,j^{\prime}} - k \} & \end{array}\right\},\end{align} $m_{1,1}, \dots, m_{1,a} $, $m_{2,1}, \dots, m_{2,b} \in \mathbb{Z}$ and $d_{\max}(n_{1,j}\gamma_1, \Lambda)$, $d_{\max}(n_{2,j}\gamma_2, \Lambda)$ as in \eqref{dmaxmore}.
\end{definition}

\begin{theorem}
\label{theorem} The set $\{ b \cdot v_{\Lambda} \;|\; b \in \frak{B}_{W(\Lambda)} \}$ is a basis for $W(\Lambda)$.
\end{theorem}

\begin{proof} First we prove that the spanning set for $W(\Lambda)$ consisting of monomial vectors \begin{align}\label{E:PBW spanning monomial vectors} bv_\Lambda= x_{n_{2,b} \gamma_2} (m_{2,b})
\cdots x_{n_{2,1}\gamma_2}(m_{2,1}) x_{n_{1,a}\gamma_1}(m_{1,a}) \cdots x_{n_{1,1}\gamma_1}(m_{1,1})v_\Lambda \end{align} can be reduced to a spanning set $\{ b \cdot v_{\Lambda} \;|\; b \in \frak{B}_{W(\Lambda)} \}$. Our argument follows closely \cite{G} by using induction on order $\prec$. The condition \begin{align*} m_{i,j+1} \leq m_{i,j} - 2n_{i,j} \ \ \textrm{if} \ \ n_{i,j+1} = n_{i,j} \end{align*} for $b \in \frak{B}_{W(\Lambda)}$ is simply the difference $2n$-condition \eqref{E:difference 2n condition for same}, and if monomial vector does not satisfy it, we may express it as linear combination of bigger monomial vectors and, by induction, we may omit it from our spanning set.

As we have seen, monomial vector \eqref{E:PBW spanning monomial vectors} is zero if the initial condition \begin{align*} m_{i,j} \leq d_{\max}(n_{i,j}\gamma_i, \Lambda) \end{align*} is not satisfied. If \begin{align*} bv_\Lambda=\dots x_{n_{1,2}\gamma_1}(m_{1,2})x_{n_{1,1} \gamma_1}(m_{1,1})v_\Lambda, \end{align*} then, by \eqref{E:difference 2n condition for different} in the second part of Remark~\ref{R:the way we use relations}, we may omit $bv_\Lambda$ from the spanning set if the condition \begin{align*} m_{1,2}\leq d_{\max}(n_{1,2}\gamma_1, \Lambda)-2n_{1,2} \end{align*} is not satisfied. By using the induction on $\prec$ we see that we may omit $bv_\Lambda$ from the spanning set if the condition \begin{align*} m_{1,j} \leq d_{\max}(n_{1,j}\gamma_1, \Lambda)-\sum_{j^{\prime} =1}\sp j 2n_{1,j} \end{align*} is not satisfied. The term $ - \sum_{j^{\prime}} \max \{ 0, n_{2,j} + n_{1,j^{\prime}} - k \}$ in the condition \begin{align*} m_{2,j} \leq d_{\max}(n_{2,j}\gamma_2, \Lambda) - \sum_{j^{\prime} < j } 2n_{2,j} - \sum_{j^{\prime}} \max \{ 0, n_{2,j} + n_{1,j^{\prime}} - k \} \end{align*} appears as a consequence of Lemma~\ref{lem_relations 4}.

What remains to prove is linear independence of the spanning set $\frak{B}_{W(\Lambda)}v_\Lambda$.
\end{proof}

\section{A proof of linear independence}

\subsection{Intertwining operators}

Here we introduce operators which we use in our proof of linear independence of quasi-particle bases. First, introduce operators \begin{align*} \mathcal{Y}(1\otimes e^{\mu},z)
= Y(1\otimes e^{\mu},z)e^{i \pi \Lambda_i}c(\cdot, \Lambda_i), \ \ \mu \in \Lambda_i + Q, \ i=1, \dots, \ell, \end{align*} where $c$ is bi-multiplicative, alternating commutator map corresponding to $\epsilon$. Operators $\mathcal{Y}(\cdot, z)$ are level one Dong-Lepowsky intertwining operators \cite{DL}. Note that for $\mu \in Q$ we have $\mathcal{Y}(1\otimes e^{\mu},z) = Y(1\otimes e^{\mu},z)$.
Let $\lambda_i = \omega_i - \omega_{i-1}$ for $i=1, \dots, \ell$. By using Jacobi identity for operators $Y(1\otimes e^{\gamma}, z)$ for $\gamma \in \Gamma$ and $\mathcal{Y}(1 \otimes e^{\lambda_i}, z)$ (cf. formula 12.8 of \cite{DL}), we see that operators $\mathcal{Y}(1 \otimes e^{\lambda_i}, z)$ commute with the action of $\tilde{\mathfrak{g}}_1$: \begin{align} \label{commutator} [Y(1\otimes e^{\gamma}, z_1),\mathcal{Y}(1 \otimes e^{\lambda_i}, z_2)]=0. \end{align} We define the following coefficients of intertwining operators (cf. \cite{P2}): \begin{align} \label{coeff_intertwining} [i]=\text{Res}\, z^{-1-\langle \lambda_i,\omega_{i-1}\rangle} {\mathcal Y}(1\otimes e^{\lambda_i},z), \quad i=1, \dots, \ell. \end{align}

\subsection{Operators $e\sp{\lambda_i}$}

Since $\langle \lambda_i, \gamma_j\rangle=\delta_{ij}$ for $i,j=1,2$, commutation formula \eqref{commutator00} implies \begin{align}\label{commutator11} x_{\gamma_j}(z)e\sp{\lambda_i}= \epsilon(\gamma_j,\lambda_i)\, z\sp{\delta_{ij}}e\sp{\lambda_i} x_{\gamma_j}(z). \end{align} By comparing coefficients we have \begin{align}\label{commutator3} x_{\gamma_j}(m)e^{\lambda_i}& =
\epsilon(\gamma_j,\lambda_i)e^{\lambda_i} x_{\gamma_j} (m + \delta_{ij}). \end{align} In particular, for $i\neq j$ operators $x_{\gamma_j}(m)$ and $e^{\lambda_i}$ commute up to a scalar $\epsilon(\gamma_j,\lambda_i)$.

\subsection{Georgiev's projection}

For fundamental module $L(\Lambda_j)$, $j\in\{0,1,2\}$, we denote the $\mathfrak h$-weight subspaces as \begin{align*} L(\Lambda_j)_{s,r}=L(\Lambda_j)_{s\gamma_2+r\gamma_1}. \end{align*} As in \cite{G}, a tensor product of $k$ fundamental modules $L(\Lambda_{j_1}),\dots,L(\Lambda_{j_k})$ we shall write as \begin{align*} L(\Lambda_{j_k})\otimes\dots\otimes L(\Lambda_{j_1}). \end{align*} For $\Lambda=k_0\Lambda_0+k_1\Lambda_1$ we consider \begin{align*} L(\Lambda)\subset L(\Lambda_1)\sp{\otimes k_1}\otimes L(\Lambda_0)\sp{\otimes k_0}, \end{align*} and for $\Lambda=k_1\Lambda_1+k_2\Lambda_2$ we consider \begin{align*} L(\Lambda)\subset L(\Lambda_1)\sp{\otimes k_1}\otimes L(\Lambda_2)\sp{\otimes k_2}. \end{align*} For such $\Lambda$'s and a given color-dual-charge type \begin{align*} (r_2\sp{(1)}, \dots, r_2\sp{(k)}; r_1\sp{(1)}, \dots, r_1\sp{(k)}) \end{align*} we define a projection $\pi=\pi_{(r_2\sp{(1)},\dots, r_1\sp{(k)})}$ from the tensor product of $k$ fundamental modules to the subspace \begin{align*} L(\Lambda_{j_k})_{r_2\sp{(1)},r_1\sp{(k)}}\otimes\dots\otimes L(\Lambda_{j_1})_{r_2\sp{(k)},r_1\sp{(1)}}. \end{align*} Since $x_\gamma(z)\sp2=0$ on every fundamental module, in the projection of a formal generating function of quasi-particle monomial vectors \begin{align*} \pi\left(x_{n_{2,b} \gamma_2} (z_{2,b}) \cdots x_{n_{2,1}\gamma_2}(z_{2,1}) x_{n_{1,a}\gamma_1}(z_{1,a}) \cdots x_{n_{1,1}\gamma_1}(z_{1,1})v_\Lambda\right) \end{align*} of the given color-dual-charge type $(r_2\sp{(1)}, \dots, r_2\sp{(k)}; r_1\sp{(1)}, \dots, r_1\sp{(k)})$ vertex operator $x_{n_{1,1}\gamma_1}(z_{1,1})=x_{\gamma_1}(z_{1,1})\sp{n_{1,1}}$ spreads over $n_{1,1}$ rightmost tensor factors: \begin{align*} 1\otimes\dots\otimes 1\otimes x_{\gamma_1}(z_{1,1})\otimes\dots\otimes x_{\gamma_1}(z_{1,1}). \end{align*} Similarly, vertex operator $x_{n_{1,2}\gamma_1}(z_{1,2})=x_{\gamma_1}(z_{1,2})\sp{n_{1,2}}$ spreads over $n_{1,2}\leq n_{1,1}$ rightmost tensor factors: \begin{align*} 1\otimes\dots\otimes 1\otimes x_{\gamma_1}(z_{1,2})\otimes\dots\otimes x_{\gamma_1}(z_{1,2}), \end{align*} and so on. Therefore, on the rightmost tensor factor we have all color $\gamma_1$ operators \begin{align*} \dots \otimes x_{\gamma_1}(z_{1,a}) \dots x_{\gamma_1}(z_{1,1})v_{\Lambda_{j_1}}, \end{align*} on the second right color $\gamma_1$ operators with charge at least two, and so on.

On the other hand, in the projection of a formal generating function of quasi-particle monomial vectors vertex operator $x_{n_{2,1}\gamma_2}(z_{2,1})$ spreads over $n_{2,1}$ leftmost tensor factors: \begin{align*} x_{\gamma_2}(z_{2,1})\otimes\dots\otimes x_{\gamma_2}(z_{2,1})\otimes1\otimes\dots\otimes 1,\end{align*} and so on. Therefore, on the leftmost tensor factor we have all color $\gamma_2$ operators \begin{align*} x_{\gamma_2}(z_{2,b}) \dots x_{\gamma_2}(z_{2,1})v_{\Lambda_{j_k}}\otimes \dots, \end{align*} on the second left color $\gamma_2$ operators with charge at least two, and so on.

\subsection{A proof of linear independence}

For $b \in \frak{B}_{W(\Lambda)}$\,, \begin{align*} b = x_{n_{2,b} \gamma_2} (m_{2,b}) \cdots x_{n_{2,1}\gamma_2}(m_{2,1}) x_{n_{1,a}\gamma_1}(m_{1,a}) \cdots x_{n_{1,1}\gamma_1}(m_{1,1}),
\end{align*} presume that $b \cdot v_{\Lambda}=0$, and consequently \begin{align} \label{linind} \pi \cdot b \cdot v_{\Lambda} = 0. \end{align} Act on \eqref{linind} with $[1]_{n_{1,1}} = 1^{\otimes (k-n_{1,1})} \otimes [1] \otimes 1^{\otimes (n_{1,1} - 1)}$ and use the fact \eqref{commutator} that $[1]$ commutes with all $x_{\gamma}(m)$ to shift $[1]$ all the way along the part of $b$ projected on $n_{1,1}-$th coordinate (from right) towards the corresponding highest weight vector. Definition \eqref{coeff_intertwining} implies \begin{align*} [1]v_{\Lambda_j} = Ce^{\lambda_1} v_{\Lambda_j}
\end{align*} for some $C\in\mathbb C\sp\times$. Now shift $e^{\lambda_1}$ all the way to the left using commutation relation \eqref{commutator11} \begin{align*} x_{\gamma_1}(z)e\sp{\lambda_1}= \epsilon(\gamma_1,\lambda_1)\, z\,e\sp{\lambda_i} x_{\gamma_j}(z). \end{align*} After dropping out the invertible operator $e^{\lambda_1}$, we finally get the altered expression with degrees of all
quasi-particles of charge $n_{1,1}$ and color $\gamma_1$ increased by one. Note that this calculation does not change degrees of quasi-particles of color $\gamma_1$ and smaller charge, because none
of its particle constituents gets projected to the coordinate of interest, or any of the quasi-particles of color $\gamma_2$, because commuting with $e^{\lambda_1}$ does not change the degrees of
particles of color $\gamma_2$.

By repeating this procedure the original degree $m_{1,1}$ will be shifted to its maximal possible value: \begin{align} \label{linind2} \pi \cdot b^{\prime} \cdot x_{n_{1,1}\gamma_1} (d_{\max}(n_{1,1}\gamma_1, \Lambda)) v_{\Lambda} = 0 \end{align} with $b^{\prime} = x_{n_{2,b} \gamma_2} (m^{\prime}_{2,b}) \cdots x_{n_{2,1}\gamma_2}(m^{\prime}_{2,1})\cdot x_{n_{1,a}\gamma_1}(m^{\prime}_{1,a}) \cdots x_{n_{1,2}\gamma_1}(m^{\prime}_{1,2})$ of same color-charge type as $b$, except for not having $n_{1,1}$. Let $n_{1,j}= n_{1,1}$ for all
$j=2, \dots, j_1$, for some $j_1 = 1,\dots, a$, and note that $j_1=1$ means that all of quasi-particles of $b^{\prime}$ of color $\gamma_1$ have charge strictly smaller than $n_{1,1}$. We may write
down the degrees of quasi-particles in $b^{\prime}$ in terms of the degrees of corresponding quasi-particles in $b$: \begin{align*} & m^{\prime}_{1,j} = \left\{ \begin{array}{l} m_{1,j} +
(d_{\max}(n_{1,1}\gamma_1, \Lambda) - m_{1,1}), \ j = 2, \dots, j_1, \\ m_{1,j}, \ j=j_1+1, \dots, a, \end{array}\right. \\ & m^{\prime}_{2,j} = m_{2,j}, \ j = 1, \dots, b.\end{align*} The fact
that $b^{\prime}$ belongs to $\frak{B}_{W(\Lambda)}$ is easily proved by checking that the degrees of its constituting quasi-particles satisfy the defining inequalities of \eqref{basis}.

Now using formulas \begin{align*} x_{\gamma}(d_{\max}(\gamma, \Lambda)) v_{\Lambda} & = C e^{\gamma} v_{\Lambda}, \ C \in \mathbb{C}^{\times}, \gamma \in \Gamma, \Lambda = \Lambda_0, \Lambda_1, \Lambda_2 \\x_{\beta_i}(m) e^{\beta_j} &= C e^{\beta_j} x_{\beta_i}(m+\langle \beta_i, \beta_j \rangle), \ C \in \mathbb{C}^{\times}, \beta_i, \beta_j \in \Gamma, \end{align*} coming out from \eqref{gaction}, \eqref{linind2} leads to \begin{align} \label{linind3} & \pi \cdot b^{\prime} \cdot x_{n_{1,1}} (d_{\max}(n_{1,1}\gamma_1, \Lambda)) v_{\Lambda} = \\ \nonumber & = C \cdot \pi \cdot b^{\prime} \cdot (1)_{n_{1,1}} v_{\Lambda} = C \cdot (1)_{n_{1,1}} \pi \cdot b^{\prime \prime} \cdot v_{\Lambda} = 0 \end{align} for some $C \in \mathbb{C}^{\times}$, with $(1)_{n_{1,1}} = 1^{\otimes(k-n_{1,1})} \otimes (e^{\gamma_1})^{\otimes n_{1,1}}$ and \begin{align*} b^{\prime \prime} = x_{n_{2,b} \gamma_2} (m^{\prime \prime}_{2,b}) \cdots x_{n_{2,1}\gamma_2}(m^{\prime \prime}_{2,1})\cdot x_{n_{1,a}\gamma_1} (m^{\prime \prime}_{1,a}) \cdots x_{n_{1,2}\gamma_1}(m^{\prime \prime}_{1,2}) \end{align*} such that \begin{align} \label{bprimeprime} & m^{\prime \prime}_{1,j} = \left\{ \begin{array}{l} m^{\prime}_{1,j} + 2n_{1,j} = m_{1,j} + (d_{\max}(n_{1,1}\gamma_1, \Lambda) - m_{1,1}) + 2n_{1,j}, \ j = 2, \dots, j_1 \\ m^{\prime}_{1,j} + 2n_{1,j} = m_{1,j} + 2n_{1,j}, \ j=j_1+1, \dots, a \end{array}\right. \\ & m^{\prime \prime}_{2,j} = \left\{ \begin{array}{l} m^{\prime}_{2,j} + n_{2,j} + n_{1,1} - k = m_{2,j}+ n_{2,j} + n_{1,1} - k, \ j = 1, \dots, j_2 \\ \nonumber m^{\prime}_{2,j} = m_{2,j}, \ j=j_2+1, \dots, b \end{array}\right., \end{align} where $j_2 = 0, \dots, b$ is such that $n_{2,j} + n_{1,1} \geq k$ for all $j=1, \dots, j_2$, with $j_2 = 0$ meaning that this inequality does not hold for any quasi-particles of color $\gamma_2$.

By checking defining inequalities of \eqref{basis} we show that $b^{\prime \prime} \in \frak{B}_{W(\Lambda)}$: first two inequalities are obvious since the described calculation treats all quasi-particles of same color and charge in the same way, hence the relative differences in degrees between neighboring quasi-particles of same charge and color are kept the same. Furthermore, for $m^{\prime \prime}_{1,j}$, $j=1, \dots, j_1$, the third inequality after introducing \eqref{bprimeprime} and taking into account $n_{1,j} =n_{1,j-1}= \dots = n_{1,1}$ becomes $m_{1,j} \leq m_{1,1} - \sum_{j^{\prime} < j} 2n_{j^{\prime},1}$, i.e. exactly the second inequality for $b$ applied successively $j-1$ times. For $j=j_1 + 1, \dots, a$, the third inequality leads directly to third inequality for $b$. The fourth inequality is checked directly and simultaneously for all $j=1, \dots, b$, let $m^{\prime \prime}_{j,2}$ be written uniquely as $m^{\prime \prime}_{2,j} = m_{2,j} + \max \{0, n_{2,j} + n_{1,1} - k \}$. This proves that $b^{\prime \prime}$ is again an element of $\frak{B}_{W(\Lambda)}$.

We now proceed with \eqref{linind3}, which after dropping out $(1)_{n_{1,1}}$ becomes \begin{align}\label{linind4} \pi \cdot b^{\prime \prime} \cdot v_{\Lambda} = 0. \end{align} If we now apply above described procedure to \eqref{linind4}, always proving that quasi-particle monomial appearing during various steps is again in $\frak{B}_{W(\Lambda)}$, after finitely many steps we arrive to the obviously false conclusion that $v_{\Lambda}=0$, which via contradiction proves $b \cdot v_{\Lambda} \neq 0$ for all $b \in \frak{B}_{W(\Lambda)}$.

If we now start with a general linear combination \begin{align} \label{lincomb} \sum_{b} C_b\cdot b \cdot v_{\Lambda}=0 \end{align} and apply the above arguments to the smallest element $b_{\min}$ of
\eqref{lincomb}, then all the other monomials of \eqref{lincomb} get annihilated at some step of the procedure described. We are therefore led to $C_{b_{\min}}\cdot b_{\min} \cdot v_{\Lambda}=0$,
which then because of $b_{\min} \cdot v_{\Lambda} \neq 0$ implies $C_{b_{\min}} = 0$. Repeating this procedure after finitely many steps proves that all coefficients of \eqref{lincomb} are zero, thus
showing the linear independence of $\frak{B}_{W(\Lambda)}$.

\begin{remark}
\label{further_directions} As we have already said, our construction is parallel to Georgiev's construction in \cite{G} of quasi-particle bases of principal subspaces of standard modules for $\mathfrak{sl}(\ell + 1, \mathbb{C})^{\widetilde{}}$. However, a major difference between the two is a nature of relations \eqref{different}. Roughly speaking, in our case different quasi-particles of small charge do not interact. For this reason we have to choose Georgiev's projection in such a way that generating functions $x_{n\gamma_1}(z)$ for quasi-particles of color $\gamma_1$ spread over $n$ tensor factors from the {\bf right}, and generating functions $x_{n\gamma_2}(z)$ for quasi-particles of color $\gamma_2$ spread over $n$ tensor factors from the {\bf left}. Such a trick is not possible in higher ranks, as easily seen on example of level $2$ for $\mathfrak{sl}(4, \mathbb{C})^{\widetilde{}}$: we have generating functions $x_{n\gamma_i}(z)$ for quasi-particles, $i=1,2,3$ and $n=1,2$. Charge one quasi-particles $x_{\gamma_i}(z)$ and $x_{\gamma_j}(z)$ do not interact for $i\neq j$, so for ``our proof'' we would need a Georgiev's projection which could place each charge one quasi-particle (altogether 3) on separate tensor factor (altogether 2).Of course, the fact that Georgiev's proof ``does not work'' for $\mathfrak{sl}(4, \mathbb{C})^{\widetilde{}}$ level $2$ does not necessarily mean that there is no fermionic formula of a corresponding form. \end{remark}

\section{Character formulas for $W(\Lambda)$}

\subsection{Character formulas}

In this section we define formal character $\chi(W(\Lambda))$ for $W(\Lambda)$ and use the quasi-particle basis $\frak{B}_{W(\Lambda)}$ given by \eqref{basis} to directly calculate the fermionic-type character formula.

\begin{definition}
For given quasi-particle monomial \begin{align} \label{monomial} b = x_{n_{2,b} \gamma_2} (m_{2,b}) \cdots x_{n_{2,1}\gamma_2}(m_{2,1}) x_{n_{1,a}\gamma_1}(m_{1,a}) \cdots x_{n_{1,1}\gamma_1}(m_{1,1}) \end{align} in $\frak{B}_{W(\Lambda)}$ define charge $c(b)$ and degree $d(b)$ by \begin{align*} c(b)= \gamma_1 \cdot \sum_{j=1}^a n_{1,j} + \gamma_2 \cdot \sum_{j=1}^b n_{2,j}, \ \ d(b) = - \sum_{j=1}^a m_{1,j} - \sum_{j=1}^b m_{2,j}. \end{align*} Note that $d(b)$ is defined with minus signs in order to be positive.

Define now the formal character $\chi(W(\Lambda))$ for $W(\Lambda)$ by \begin{align} \label{character} \chi(W(\Lambda))(z_1, z_2; q) = \sum_{n_1, n_2 \geq 0} A_{\Lambda}^{n_1,n_2}(q) z_1^{n_1} z_2^{n_2}, \end{align} where $A_{\Lambda}^{n_1,n_2}(q)$ encodes the number of basis elements in component of $W(\Lambda)$ spanned by quasi-particle of charge $n_1\gamma_1 + n_2 \gamma_2$. \end{definition}

For fixed $n_1, n_2 \geq 0$ and non-negative integers $M_{i,j}$ such that $\sum_{j=1}^k M_{i,j}= n_i$, $i=1,2$, $j=1, \dots, k$, we look at all $b$ in $\frak{B}_{W(\Lambda)}$ given by \eqref{monomial} of charge $c(b)= n_1\gamma_1 + n_2 \gamma_2$ and having $M_{i,j}$ quasi-particle constituents with charge $j$ and color $\gamma_i$. From the way $\frak{B}_{W(\Lambda)}$ is given in \eqref{basis} we
have \begin{align} \label{dimensions} A_{\Lambda}^{n_1,n_2}(q) = \sum_{\genfrac{}{}{0pt}{}{ \sum_{j=1}^k jM_{1,j}= n_1}{\sum_{j=1}^k jM_{2,j}= n_2 }} \frac{q^{d(b_{M_{i,j}})}}{\prod_{i=1}^2 \prod_{j=1}^k (q)_{M_{i,j}}}, \end{align} where $b_{M_{i,j}}$ is the so-called minimal monomial in $\frak{B}_{W(\Lambda)}$, characterized by having smallest possible degree of all basis elements given the above constraints. Therefore, thanking the fact that character merely reflects the way $\frak{B}_{W(\Lambda)}$ is presented in \eqref{basis}, the task of finding the character formula reduces to
relatively simple calculation of the degree of minimal monomial.

We will write \begin{align*} b_{M_{i,j}} = \prod_{j=1}^k \prod_{s_{2,j} = N_{2,k-j}+1}^{N_{2,k-j+1}} x_{j\gamma_2}(m_{2,s_{2,j}}^{\max}) \cdot \prod_{j=1}^k \prod_{s_{1,j} = N_{1,j+1}+1}^{N_{1,j}} x_{j\gamma_1}(m_{1,s_{1,j}}^{\max}), \end{align*} where for $j=1, \dots, k$ the following partial sums are given: \begin{align} \label{connection} N_{1,j} & = M_{1,j} + \dots + M_{1,k}, \ \ N_{1,k+1} = 0 \\ \nonumber N_{2,j} & = M_{2,k-j+1} + \dots + M_{2,k}, \ \ N_{2,0}=0. \end{align}

Now, using first and third inequality of \eqref{basis} we calculate for fixed $j=1, \dots, k$: \begin{align*} m_{1,N_{1,j+1}+1}^{\max} & = d_{\max}(j\gamma_1, \Lambda) - \sum {2j} = d_{\max}(j\gamma_1, \Lambda) - 2j(M_{1,j+1} + \dots + M_{1,k}) \\ m_{1,s_{1,j}}^{\max} & = m_{1,s_{1,j} -1}^{\max} - 2j = \dots = m_{1,N_{1,j+1}+1}^{\max} - 2j(s_{1,j}-N_{1,j+1}-1) \end{align*} for $s_{1,j} > N_{1,j+1}+1$, and therefore \begin{align} \label{gamma1} & d(\prod_{s_{1,j} = N_{1,j+1}+1}^{N_{1,j}} x_{j\gamma_1}(m_{1,s_{1,j}}^{\max})) = - \sum_{s_{1,j} = N_{1,j+1}+1}^{N_{1,j}} m_{1,s_{1,j}}^{\max} = \\ \nonumber &= 2j(1+ \dots + M_{1,j} -1) - M_{1,j}(d_{\max}(j\gamma_1, \Lambda) - 2j(M_{1,j+1} + \dots + M_{1,k})) = \\ \nonumber & = jM_{1,j}^2 + 2j\cdot M_{1,j} \cdot (M_{1,j+1}+ \dots + M_{1,k}) + \max \{ 0,j-k_0 \} \cdot M_{1,j}. \end{align}

Similarly, second and fourth inequality of \eqref{basis} give \begin{align*} m_{2,N_{2,k-j}+1}^{\max} & = d_{\max}(j\gamma_2, \Lambda) - 2j(M_{2,j+1} + \dots + M_{2,k})- \\ & - (\max \{0, j+1-k \}\cdot M_{1,1} + \dots + \max \{ 0, j\} \cdot M_{1,k}) = \\ & = d_{\max}(j\gamma_2, \Lambda) - 2j(M_{2,j+1} + \dots + M_{2,k})- \\ &- (M_{1,k-j+1} + \dots + jM_{1,k}) \\ m_{2,s_{2,j}}^{\max} & = m_{2,s_{2,j} -1}^{\max} - 2j = \dots = m_{2,N_{2,k-j+1}+1}^{\max} - 2j(s_{2,j}-N_{2,k-j}) \end{align*} for $s_{2,j} > N_{2,k-j}+1$, so \begin{align} & \label{gamma2} d(\prod_{s_{2,j} = N_{2,k-j}+1}^{N_{2,k-j+1}} x_{j\gamma_2}(m_{2,s_{2,j}}^{\max})) = - \sum_{s_{2,j} = N_{2,k-j}+1}^{N_{2,k-j+1}} m_{2,s_{2,j}}^{\max} = \\ \nonumber &= 2j(1+ \dots + M_{2,j} -1) - M_{2,j}(d_{\max}(j\gamma_2, \Lambda) - 2j(M_{2,j+1} + \dots + M_{2,k})) - \\ \nonumber & - (M_{1,k-j+1}+ \dots + jM_{1,k}) = \\ \nonumber & = jM_{2,j}^2 + M_{2,j} \cdot (2jM_{2,j+1}+ \dots + 2jM_{2,k} + M_{1,k-j+1} + \dots + jM_{1,k}) + \\ \nonumber & + \max \{ 0,j-(k_0+k_1) \} \cdot M_{2,j}. \end{align}

Introduce now $2k \times 2k$ matrix $A$ defined by \begin{align} \label{matrix} Q^{(k)} = \left(\begin{array}{c|c} A^{(k)} & B^{(k)} \\ \hline 0 & A^{(k)} \end{array} \right), \end{align}
where $A^{(k)}= (\min \{ i,j \})_{i, j = 1}^{k}$, $B^{(k)} = (\max \{ 0, i+j-k \})_{i, j = 1}^{k}$, and \begin{align} \label{vector} \mathbf{L}^{(k)}_{\Lambda}:= (\underbrace{0, \dots, 0}_{k_0}, 1, 2, \dots , k-k_0, \underbrace{0, \dots , 0}_{k_0 + k_1},1, \dots, k_2). \end{align} Then for $\mathbf{M}:={}^{t}(M_{1,1}, \dots , M_{1,k}, M_{2,1}, \dots , M_{2,k})$ and using \eqref{gamma1} and \eqref{gamma2} it easy to check that we have \begin{align*} d(b_{M_{i,j}}) = {}^{t}\mathbf{M}\cdot Q^{(k)} \cdot \mathbf{M} + \mathbf{L}^{(k)}_{\Lambda}\cdot \mathbf{M}, \end{align*} so from \eqref{character} and \eqref{dimensions} we now have the following result:

\begin{theorem}
\label{formulas} For $\Lambda = k_0 \Lambda_0 + k_1 \Lambda_1$ or $\Lambda = k_1 \Lambda_1 + k_2 \Lambda_2$ being the highest weight of standard $\mathfrak{sl}(3,\mathbb{C})^{\widetilde{}}$-module
$L(\Lambda)$, the following fermionic-type formula holds for the formal character $\chi(W(\Lambda))$ of Feigin-Stoyanovsky's type subspace $W(\Lambda)$ of $L(\Lambda)$: \begin{align} \label{formula} \chi(W(\Lambda))(z_1, z_2; q) = \sum_{n_1, n_2 \geq 0} \sum_{\genfrac{}{}{0pt}{}{ \sum_{j=1}^k jM_{1,j} = n_1}{\sum_{j=1}^k jM_{2,j}= n_2 }} \frac{q^{{}^{t}\mathbf{M}\cdot Q^{(k)} \cdot \mathbf{M} + \mathbf{L}^{(k)}_{\Lambda}\cdot \mathbf{M}}} {\prod_{i=1}^2 \prod_{j=1}^k (q)_{M_{i,j}}} z_1^{n_1} z_2^{n_2}, \end{align} with $Q^{(k)}$ given by \eqref{matrix}, $\mathbf{L}^{(k)}_{\Lambda}$ by \eqref{vector}, and $\mathbf{M}={}^{t}(M_{1,1}, \dots , M_{1,k}, M_{2,1}, \dots , M_{2,k})$.
\end{theorem}

\begin{example}
We present fermionic-type formulas of Theorem \ref{formulas} in the case of Feigin-Stoyanovsky's type subspaces $W(\Lambda)$ of standard level two $\mathfrak{sl}(3,\mathbb{C})^{\widetilde{}}-$modules. Note that the Theorem does not provide us with formula in the case of $\Lambda = \Lambda_0 + \Lambda_2$, while for all other cases the formulas are obtained. First, we get \begin{align*} Q^{(2)} = \left( \begin{array}{cccc} 1 & 1 & 0 & 1 \\ 1 & 2 & 1 & 2 \\ 0 & 0 & 1 & 1 \\ 0 & 0 & 1 & 2 \end{array} \right), \end{align*} so the quadratic term reads \begin{align*} Q &= M_{11}^2 + 2M_{12}^2 + M_{21}^2 + 2M_{22}^2 + \\ & + 2M_{11} M_{12} + M_{11}M_{22} + M_{12}M_{21} + 2M_{12}M_{22} + 2M_{21}M_{22}. \end{align*} The linear terms $L_{\Lambda} = L_{k_0,k_1,k_2}$ for $\Lambda = k_0\Lambda_0 + k_1\Lambda_1 + k_2\Lambda_2$ have the following values: \begin{align*} L_{2,0,0} & = 0 \\ L_{1,1,0} & = M_{12} \\ L_{0,2,0} & = M_{11} + 2M_{12} \\ L_{0,1,1} & = M_{11} + 2M_{12} + M_{22} \\ L_{0,0,2} & = M_{11} + 2M_{12} + M_{21} + 2M_{22}. \end{align*} Finally, we get \begin{align*} \chi(W(\Lambda))(z_1, z_2; q) & = \sum_{n_1, n_2 \geq 0} \sum_{\genfrac{}{}{0pt}{}{M_{11} + 2M_{12} = n_1}{M_{21} + 2M_{22} = n_2}} \frac{q^{Q + L_{\Lambda}}}{ (q)_{M_{11}} (q)_{M_{12}} (q)_{M_{21}} (q)_{M_{22}}} z_1^{n_1} z_2^{n_2}. \end{align*}
\end{example}

\subsection{Comparison with existing results}

From \eqref{connection} we get $M_{1,j}=N_{1,j} - N_{1,j+1}$, $M_{2,k-j+1} = N_{2,j} - N_{2,j-1}$, $j=1,\dots, k$, which introduced in \eqref{formula} easily give \begin{align} \label{formula2} & \chi(W(\Lambda))(z_1, z_2; q) = \\ \nonumber &= \sum_{n_1, n_2 \geq 0} \sum_{ \genfrac{}{}{0pt}{}{\genfrac{}{}{0pt}{}{\sum_{i=1}^k N_{1,i} = n_1}{N_{1,1} \geq \cdots \geq N_{1,k} \geq 0}}{\genfrac{}{}{0pt}{}{\sum_{i=1}^k N_{2,i} = n_2}{N_{2,k} \geq \cdots \geq N_{2,1} \geq 0}}} \frac{q^{{}^{t}\mathbf{N} \cdot {Q^{\prime}}^{(k)} \cdot \mathbf{N} + \mathbf{L^{\prime}}^{(k)}_{\Lambda} \cdot \mathbf{N}}}{\prod_{j=1}^k (q)_{N_{1,j} - N_{1,j+1}} \prod_{j=1}^k (q)_{N_{2,j} - N_{2,j-1}}} z_1^{n_1} z_2^{n_2},\end{align} where ${Q^{\prime}}^{(k)} = {}^{t}R^{(k)}\cdot Q^{(k)}\cdot R^{(k)}$ for \begin{align*} R^{(k)} = \left(\begin{array}{c|c} C^{(k)} & 0 \\ \hline 0 & D^{(k)} \end{array} \right), \end{align*} $C^{(k)} = (C_{i,j}^{(k)})_{i,j=1}^k$, $D^{(k)} = (D_{i,j}^{(k)})_{i,j=1}^k$,
\begin{align*} C_{i,j}^{(k)} = \left\{ \begin{array}{ll} 1, \ j=i \\ -1, \ j=i+1 \\ 0, \ \textrm{otherwise} \end{array}\right. \ \ \ \ \ D_{i,j}^{(k)} = \left\{ \begin{array}{ll} 1, \ j=k-i+1 \\ -1, \ j=k-i \\ 0, \ \textrm{otherwise} \end{array}\right., \end{align*} and $\mathbf{L^{\prime}}^{(k)}_{\Lambda} = R^{(k)}\cdot \mathbf{L}^{(k)}_{\Lambda}$, $\mathbf{N}={}^{t}(N_{1,1}, \dots , N_{1,k}, N_{2,1}, \dots , N_{2,k})$.

It is not hard to see that formulas \eqref{formula2} exactly match the corresponding ones obtained in Theorem 3.11 in \cite{J3}. This basically suggests that, although it was the combinatorial bases parameterized by $(k,3)-$admissible configurations that were used in obtaining recurrence relations for characters and character formulas, the underlying logic of the formulas themselves was that of quasi-particles.

\begin{remark}
We can also write character formulas in terms of summation over all color-dual-charge types, as it was originally written in \cite{G} for fermionic characters of principal subspaces: \begin{align} \label{formula3}& \chi(W(k_0\Lambda_0 + k_1\Lambda_1))(z_1,z_2;q) = \sum_{r_{1}^{(1)}\geq \ldots \geq r_{1}^{(k)} \geq 0}\; \frac{q^{{r_{1}^{(1)}}^{2}+ \ldots + {r_{1}^{(k)}}^{2} }}{(q)_{r_{1}^{(1)} - r_{1}^{(2)}} \ldots (q)_{r_{1}^{(k - 1)} - r_{1}^{(k)}} (q)_{r_{1}^{(k)}}} \\ \nonumber & \sum_{r_{2}^{(1)}\geq \ldots \geq r_{2}^{(k)} \geq 0} \frac{q^{{r_{2}^{(1)}}^{2} + \ldots + {r_{2}^{(k)}}^{2} + r_{2}^{(1)} r_{1}^{(k)} + \ldots + r_{2}^{(k)}r_{1}^{(1)} + r_1^{(k_0+1)} + \dots + r_1^{(k)}}}{(q)_{r_{2}^{(1)} - r_{2}^{(2)}} \ldots (q)_{r_{2}^{(k - 1)} - r_{2}^{(k)}} (q)_{r_{2}^{(k)}}} z_1\sp{r_1} z_2\sp{r_2} \\ \nonumber & \chi(W(k_1\Lambda_1 + k_2\Lambda_2))(z_1,z_2;q) = \sum_{r_{1}^{(1)}\geq \ldots \geq r_{1}^{(k)} \geq 0}\; \frac{q^{{r_{1}^{(1)}}^{2}+ \ldots + {r_{1}^{(k)}}^{2} }}{(q)_{r_{1}^{(1)} - r_{1}^{(2)}} \ldots (q)_{r_{1}^{(k - 1)} - r_{1}^{(k)}} (q)_{r_{1}^{(k)}}} \\ \nonumber & \sum_{r_{2}^{(1)}\geq \ldots \geq r_{2}^{(k)} \geq 0} \frac{q^{{r_{2}^{(1)}}^{2} + \ldots + {r_{2}^{(k)}}^{2} + r_{2}^{(1)} r_{1}^{(k)} + \ldots + r_{2}^{(k)}r_{1}^{(1)} + r_1^{(1)} + \dots + r_1^{(k)} + r_2^{(k+1-k_2)} + \dots + r_2^{(k)}}}{(q)_{r_{2}^{(1)} - r_{2}^{(2)}} \ldots (q)_{r_{2}^{(k - 1)} - r_{2}^{(k)}} (q)_{r_{2}^{(k)}}} z_1\sp{r_1} z_2\sp{r_2} \end{align} for $r_1 = r_{1}^{(1)} + \dots + r_{1}^{(k)}$ and $r_2 = r_{2}^{(1)} + \dots + r_{2}^{(k)}$.
Although here we used notation introduced in Definition~\ref{cct} in order to be compatible with \cite{G}, note that taking $r_1^{(k)} = N_{1,j}$ and $r_2^{(j)}= N_{2,k-j+1}$ for $j=1, \dots, k$, turns formulas \eqref{formula3} into \eqref{formula2}.
\end{remark}

\end{document}